\documentclass[a4paper]{amsart}
\usepackage{graphicx}
\usepackage{amssymb}
\usepackage{amsmath}
\usepackage{amsthm}
\usepackage[all]{xy}

\newtheorem{thm}{Theorem}[section]

\newtheorem{lem}[thm]{Lemma}

\newtheorem{prop}[thm]{Proposition}
\theoremstyle{definition}
\newtheorem{defn}[thm]{Definition}
\theoremstyle{remark}
\newtheorem{rem}[thm]{\bf Remark}
\newtheorem{exm}[thm]{\bf Example}

\begin{document}

\title{algebra extensions and derived-discrete algebras}

\author{Jie Li}

\subjclass[2010]{16G10, 16E35}
\keywords{derived-discrete algebra, split extension, separable extension.}
\thanks{E-mail: lijie0$\symbol{64}$mail.ustc.edu.cn}

\begin{abstract}
	Let $\phi\colon A\rightarrow B$ be an algebra homomorphism between finite dimensional algebras. We prove that if $\phi$ is split, the derived-discreteness of $A$ implies the derived-discreteness of $B$; if $\phi$ is separable and the right $A$-module $B_A$ is projective, the converse holds. We prove an analogous statement for piecewise hereditary algebras. 
\end{abstract}

\maketitle

\section{introduction}
The notion of a derived-discrete algebra is introduced by \cite{V}. This class of algebras plays a special role in the representation theory of algebras, since their derived categories are accessible \cite{NPP}. The derived-discrete algebras over an algebraically closed field were classified by \cite{V} up to Morita equivalence, and by \cite{BGS} up to derived equivalence. 

We plan to classify derived-discrete algebras over an arbitrary infinite field. One approach is to investigate the relation between derived-discreteness and field extensions. More generally, we study the relation between derived-discreteness and algebra extensions.

Let us describe our main results. Assume that $k$ is an infinite field, and that $A$ and $B$ are finite dimensional algebras over $k$. Let $\phi\colon A\rightarrow B$ be an algebra extension, that is, an $k$-algebra homomorphism. We prove that if $\phi$ is a split extension, the derived-discreteness of $B$ implies the derived-discreteness of $A$; if $\phi$ is a separable extension and the right $A$-module $B_A$ is projective, the derived-discreteness of $A$ implies the derived-discreteness of $B$; see Theorem \ref{main}. The first statement strengthens \cite[3.3 Proposition]{V}. The condition that $B_A$ is projective as a right $A$-module in the second statement is necessary; see Example \ref{cntemp}. We prove analogous statements for piecewise hereditary algebras; see Proposition \ref{ph}. 

Recall from \cite[2.1 Theorem]{V} the classification of derived-discrete algebras: a connected algebra $A$ over an algebraically closed field $k$ is derived-discrete if and only if $A$ is either piecewise hereditary of Dynkin type or Morita equivalent to a gentle one-cycle algebra with the clock condition. As an application of the above results, we prove that the dichotomic classification of derived-discrete algebras is compatible with skew group algebra extensions; see Proposition~\ref{last}.

Throughout, we fix an infinite field $k$, not algebraically closed in general. We require that all the algebras are finite dimensional over $k$, and all the functors are $k$-linear.

\section{Derived-discrete algebras}
In this section, we recall derived-discrete algebras. Denote by $\mathbb{N}^{(\mathbb{Z})}$ the set of vectors $\underline{n}=(n_i)_{i\in\mathbb{Z}}$ of natural numbers with only finitely many nonzero entries.

Let $A$ be a finite dimensional algebra over $k$. Denote by $A\mbox{-mod}$ the abelian category of finitely generated left $A$-modules. Let
$\mathbf{D}^b(A\mbox{-mod})$ be its bounded derived category. 
For each $X$ in $\mathbf{D}^b(A\mbox{-{\rm mod}})$, let $$\underline{\dim}_kX=(\dim_kH^i(X))_{i\in\mathbb{Z}}\in{\mathbb{N}^{(\mathbb{Z})}}$$ be its \textit{cohomology dimension vector}. Denote by $[X]$ the isomorphism class of $X$ in $\mathbf{D}^b(A\mbox{-{\rm mod}})$.

\begin{defn}
	A finite dimensional $k$-algebra $A$ is called \textbf{derived-discrete} over $k$, if for any vector $\underline{n}=(n_i)_{i\in\mathbb{Z}}$ in $\mathbb{N}^{(\mathbb{Z})}$, $$\{[X]\in\mathbf{D}^b(A\mbox{-{\rm mod}})\,|\,\underline{\dim}_kX=\underline{n}\}$$ is a finite set.
\end{defn}

The lemma below shows that the definition above is equivalent to the one in \cite[1.1]{V}. Denote by $K_0(A)$ the Grothendieck group of $A$-mod. For each $X$ in $\mathbf{D}^b(A\mbox{-{\rm mod}})$, set $K\underline{\dim} X=(\textbf{m}_i)_{i\in\mathbb{Z}}\in K_0(A)^{(\mathbb{Z})}$, where $\textbf{m}_i$ is the dimension vector of $H^i(X)$.

\begin{lem}
	A finite dimensional $k$-algebra $A$ is derived-discrete if and only if for each $(\textbf{m}_i)_{i\in\mathbb{Z}}\in K_0(A)^{(\mathbb{Z})}$, 
	$$\{[X]\in\mathbf{D}^b(A\mbox{-{\rm mod}})\,|\,\mbox{ X is indecomposable with }K\underline{\dim} X=(\textbf{m}_i)_{i\in\mathbb{Z}}\}$$ is a finite set.
\end{lem}	
\begin{proof}
	For each $(\textbf{m}_i)_{i\in\mathbb{Z}}\in K_0(A)^{(\mathbb{Z})}$, let $n_i$ be the total dimension of $\textbf{m}_i$. For each object $X$ in $\mathbf{D}^b(A\mbox{-{\rm mod}})$ such that $H^i(X)=\textbf{m}_i$, $\forall i\in\mathbb{Z}$, we have $\underline{\dim}_kX=(n_i)_{i\in\mathbb{Z}}=\underline{n}$. Hence the ``only if'' part holds.
		
	Conversely, we assume that $$\{[X]\in\mathbf{D}^b(A\mbox{-{\rm mod}})\,|\,\mbox{ X is indecomposable with }K\underline{\dim} X=(\textbf{m}_i)_{i\in\mathbb{Z}}\}$$ is a finite set for each $(\textbf{m}_i)_{i\in\mathbb{Z}}$. For each $(n_i)_{i\in\mathbb{Z}}\in\mathbb{N}^{(\mathbb{Z})}$, there are only finitely many $(\textbf{m}_i)_{i\in\mathbb{Z}}\in K_0(A)^{(\mathbb{Z})}$ such that the total dimension of $\textbf{m}_i$ equals $n_i$. Hence  $$\{[X]\in\mathbf{D}^b(A\mbox{-{\rm mod}})\,|\,X \mbox{ is indecomposable with }\underline{\dim}_kX=\underline{n}\}$$ is a finite set. Since $\mathbf{D}^b(A\mbox{-{\rm mod}})$ is Krull-Schmidt, the ``if'' part holds.
\end{proof}

\begin{lem}
	Let $K/k$ be a finite field extension and $A$ be a finite dimensional $K$-algebra. Then $A$ is derived-discrete over $K$ if and only if it is derived-discrete over $k$.
\end{lem}
\begin{proof}
	Assume that $K/k$ is a finite field extension of degree $l$. For each $X\in\mathbf{D}^b(A\mbox{-{\rm mod}})$, we have $$\underline{\dim}_kX=(\dim_kH^i(X))_{i\in\mathbb{Z}}=(l\cdot\dim_KH^i(X))_{i\in\mathbb{Z}}=l\cdot\underline{\dim}_KX.$$ Hence for each $\underline{n}=(n_i)_{i\in\mathbb{Z}}$ in $\mathbb{N}^{(\mathbb{Z})}$, $$(\divideontimes)\mbox{ }\{[X]\in\mathbf{D}^b(A\mbox{-{\rm mod}})\,|\,\underline{\dim}_KX=\underline{n}\}=\{[X]\in\mathbf{D}^b(A\mbox{-{\rm mod}})\,|\,\underline{\dim}_kX=l\cdot\underline{n}\}. $$ Therefore, $A$ is derived-discrete over $K$ if it is derived-discrete over $k$. 
	
	Conversely, for each $\underline{m}=(m_i)_{i\in\mathbb{Z}}$ in $\mathbb{N}^{(\mathbb{Z})}$, the set $$\{[X]\in\mathbf{D}^b(A\mbox{-{\rm mod}})\,|\,\underline{\dim}_kX=\underline{m}\}$$ is non-empty only when $m_i$ is divisible by $l$ for any $i\in\mathbb{Z}$. By $(\divideontimes)$ again, $A$ is derived-discrete over $k$ if it is derived-discrete over $K$. 
\end{proof}
 
Denote by $\mathbf{K}^-(A{\mbox{-\rm proj}})$ the homotopy category of bounded-above complexes of finitely generated projective left $A$-modules. Let $\mathbf{K}^b(A{\mbox{-\rm proj}})$ (\emph{resp}. $\mathbf{K}^{-,b}(A{\mbox{-\rm proj}})$) be its full subcategory consisting of bounded complexes (\emph{resp}. complexes with bounded cohomologies). There is a well-known triangle equivalence $$p\colon \mathbf{D}^b(A{\mbox{-\rm mod}})\longrightarrow\mathbf{K}^{-,b}(A{\mbox{-\rm proj}}),$$ sending $X$ to its projective resolution $pX$. For each $P$ in $\mathbf{K}^-(A{\mbox{-\rm proj}})$, by $P_{\geq t}$ in $\mathbf{K}^b(A{\mbox{-\rm proj}})$ we denote the brutal truncation of $X$ at degree $t$.

\begin{lem}\label{iso}
	Let $X$, $Y$ be in $\mathbf{D}^b(A\mbox{-{\rm mod}})$, and $t$ be an integer such that $H^{i}(X)=H^{i}(Y)=0$ whenever $i\leq t$. Then $(pX)_{\geq t}\simeq (pY)_{\geq t}$ in $\mathbf{K}^b(A{\mbox{-\rm proj}})$ implies that $X\simeq Y$ in $\mathbf{D}^b(A\mbox{-{\rm mod}})$.
\end{lem}
\begin{proof}
	Let $(f^i)_{i\in \mathbb{Z}}\colon (pX)_{\geq t}\rightarrow (pY)_{\geq t}$ in $\mathbf{K}^b(A{\mbox{-\rm proj}})$ be a homotopy equivalence. By assumption, $H^{t}(pX)=H^{t}(pY)=0$. So $(f^i)_{i\in \mathbb{Z}}$ induces an quasi-isomorphism $(\overline{f}^i)_{i\in \mathbb{Z}}$ from $\overline{pX}$ to $\overline{pY}$ as follows.\\
	$\xymatrix@C=5ex{
		\overline{pX}\colon\dots\ar[r] & 0\ar[d]_{0}\ar[r] & {\rm Ker}d^{t}_{pX}\ar[d]_{\overline{f}^{t-1}} \ar[r] & (pX)^{t}\ar[r]^{d^{t}_{pX}}\ar[d]_{\overline{f}^t=f^{t}} & (pX)^{t+1}\ar[d]^{\overline{f}^{t+1}=f^{t+1}}\ar[r]^{d^{t+1}_{pX}} & (pX)^{t+2}\ar[d]^{\overline{f}^{t+2}=f^{t+2}}\ar[r]^{d^{t+2}_{pX}}&\dots\\
		\overline{pY}\colon\dots\ar[r] & 0\ar[r] &{\rm Ker}d^{t}_{pY}\ar[r] & (pY)^{t}\ar[r]_{d^{t}_{pY}} & (pY)^{t+1}\ar[r]_{d^{t+1}_{pY}} &(pY)^{t+2}\ar[r]_{d^{t+2}_{pY}}&\dots
	}$\\
	So we have isomorphisms $X\simeq pX\simeq\overline{pX}\simeq\overline{pY}\simeq pY\simeq Y$ in $\mathbf{D}^b(A\mbox{-{\rm mod}})$.
\end{proof}

Recall that a complex $(P^i,d^i)$ in $\mathbf{K}^-(A\mbox{-proj})$ is called \emph{homotopically-minimal} if ${\rm Im}d^i\subseteq {\rm rad}P^{i+1}$ for each $i$. Each $X$ in $\mathbf{D}^{b}(A\mbox{-mod})$ has a homotopically-minimal projective resolution which is quasi-isomorphic to $X$; see \cite[Proposition B.1]{CYZ}.

\begin{lem}\label{bound}
	Assume that we are given $\underline{n}=(n_i)_{i\in\mathbb{Z}}\in\mathbb{N}^{(\mathbb{Z})}$. Then the set $$p_i:=\{\dim_kP^i\,|\,P\in\mathbf{K}^{-,b}(A\mbox{-{\rm proj}})\mbox{ homotopically-minimal with }\underline{\dim}_kP=\underline{n}\}$$ is bounded for each $i\in\mathbb{Z}$.
\end{lem}
\begin{proof}
	By assumption, for each homotopically-minimal $P$ in $\mathbf{K}^-(A\mbox{-{\rm proj}})$, we have
	\begin{align}
	\dim_k(P^i/{{\rm rad}P^i})\leq &\dim_k(P^i/{\rm Im}d^{i-1})\notag\\
	={}& \dim_kP^i-\dim_k{\rm Im}d^{i-1}\notag\\
	={}&\dim_k{\rm Ker}d^i+\dim_k{\rm Im}d^i-\dim_k{\rm Im}d^{i-1}\notag\\
	={}&\dim_kH^i(P)+\dim_k{\rm Im}d^i\notag\\
	\leq{}&\dim_kH^i(P)+\dim_kP^{i+1}.\notag
	\end{align}
	For each homotopically-minimal $P$ in $\mathbf{K}^-(A\mbox{-{\rm proj}})$ with $\underline{\dim}_kP=\underline{n}$, let $r$ be the largest integer such that $n_r\neq 0$. Then $r$ is also the largest number such that $P^r\neq0$. So $\dim_kP^i=0$ for $i>r$. Recall a fact that, given $n$ in $\mathbb{N}$, the set $$\{\dim_kQ\,|\,Q\mbox{ a projective } A\mbox{-module with }\dim_k(Q/{\rm rad}Q)\leq n\}$$ is bounded. Hence $p_r$ is bounded since $\dim_k(P^r/{{\rm rad}P^r})\leq n_r$. 
	
	Once $p_{t+1}$ is bounded for some $t\leq r-1$, the set $$\{\dim_k(P^{t}/{{\rm rad}P^{t}})\,|\,P\in\mathbf{K}^{-,b}(A\mbox{-{\rm proj}})\mbox{ homotopically-minimal with }\underline{\dim}_kP=\underline{n}\}$$ is bounded by the inequality above. Then $p_{t}$ is bounded by the fact. Inductively, we can prove the statement.
\end{proof}

Recall that the \textit{component dimension vector} of a bounded complex $X$ is denoted by $$\mbox{c-}\underline{\dim}_kX=(\dim_kX^i)_{i\in\mathbb{Z}}\in\mathbb{N}^{(\mathbb{Z})}.$$

The following lemma is essentially contained in \cite[Theorem 2.1 (\romannumeral2) and (\romannumeral3)]{V} and \cite[Theorem 2.3 a)]{B}. We include a direct proof.

\begin{lem}\label{spequ}
	The following statements are equivalent. \\
	{\rm (1)} The algebra $A$ is derived-discrete over $k$.\\
	{\rm (2)} For each $\underline{n}\in\mathbb{N}^{(\mathbb{Z})}$, $\{[P]\in\mathbf{K}^b(A\mbox{-{\rm proj}})\,|\,\underline{\dim}_kP=\underline{n}\}$ is a finite set.\\
	{\rm (3)} For each $\underline{n}\in\mathbb{N}^{(\mathbb{Z})}$, $\{[P]\in\mathbf{K}^b(A\mbox{-\rm{proj}})\,|\,\mbox{c-}\underline{\dim}_kP=\underline{n}\}$ is a finite set.
\end{lem}
\begin{proof}
	(1)$\Rightarrow$(2) is obvious.
	
	(2)$\Rightarrow$(3). For each $\underline{n}\in\mathbb{N}^{(\mathbb{Z})}$, $\underline{n}$ has finitely many partitions. By assumption, the set $$\{[P]\in\mathbf{K}^b(A\mbox{-{\rm proj}})\,|\,\underline{\dim}_kP\leq\underline{n}\}$$ is finite. Since the cohomology dimension vector is not larger than the component dimension vector, the set $$\{[P]\in\mathbf{K}^b(A\mbox{-proj})\,|\,\mbox{c-}\underline{\dim}_kP=\underline{n}\}$$ is finite. 
	
	(3)$\Rightarrow$(1). For each $\underline{n}=(n_i)_{i\in\mathbb{Z}}\in\mathbb{N}^{(\mathbb{Z})}$, let $t$ be the least number such that $n_{t+1}\neq0$ and $r$ be the largest number such that $n_r\neq0$. For each $X\in\mathbf{D}^{b}(A\mbox{-mod})$, let $pX\in\mathbf{K}^-(A\mbox{-proj})$ be the homotopically-minimal projective resolution. By Lemma~\ref{bound}, for each $i$, $\dim_k(pX)_{\geq t}^i$ is uniformly bounded, say by $m_i$. We can assume that $m_i=0$ for $i<t$ and $i>r$. Set $\underline{m}=(m_i)_{i\in\mathbb{Z}}\in\mathbb{N}^{(\mathbb{Z})}$.
	
	Notice that $\underline{m}$ has finitely many partitions. By assumption, the set $$\{[(pX)_{\geq t}]\in\mathbf{K}^b(A\mbox{-proj})\,|\,X\in\mathbf{D}^{b}(A\mbox{-mod})\mbox{ with c-}\underline{\dim}_k(pX)_{\geq t}\leq\underline{m}\}$$ is finite. By the argument in the above paragraph, the set $$\{[(pX)_{\geq t}]\in\mathbf{K}^b(A\mbox{-proj})\,|\,X\in\mathbf{D}^{b}(A\mbox{-mod})\mbox{ with }\underline{\dim}_kX=\underline{n}\}$$ is finite. By Lemma \ref{iso}, the set $$\{[X]\in\mathbf{D}^b(A\mbox{-{\rm mod}})\,|\,\underline{\dim}_kX=\underline{n}\}$$ is finite. 
\end{proof}

\section{Separable functors and algebra extensions}
In this section, we recall the notions of separable functors, separable extensions and split extensions. 

According to \cite{NBO}, a functor $F\colon \mathcal{C}\rightarrow \mathcal{D}$ is called \emph{separable} if for any $X, Y$ in $\mathcal{C}$, there is a map $$H_{X,Y}\colon {\rm Hom}_\mathcal{D}(F(X),F(Y))\rightarrow{\rm Hom}_\mathcal{C}(X,Y)$$ such that $H_{X,Y}(F(f))=f$, for any $f\in{\rm Hom}_\mathcal{C}(X,Y)$, and $H_{X,Y}$ is natural in $X$ and $Y$.
It is called a cleaving functor in \cite{V}. 

Let $\phi\colon  A\rightarrow B$ be a $k$-algebra homomorphism (in some literature, it is called a $k$-algebra extension). It induces the restriction functor $${\rm Hom}_B(B,-)\colon B\mbox{-mod} \rightarrow A\mbox{-mod},$$ and its left adjoint functor $$B\otimes_A-\colon A\mbox{-mod}\rightarrow B\mbox{-mod}.$$

\begin{defn}[{\cite[1.3]{NBO}}]
	We call an algebra homomorphism $\phi\colon A\rightarrow B$ a \textbf{split} algebra extension if $B\otimes_A-\colon A\mbox{-mod}\rightarrow B\mbox{-mod}$ is separable, and a \textbf{separable} algebra extension if ${\rm Hom}_B(B,-)\colon B\mbox{-mod}\rightarrow A\mbox{-mod}$ is separable.
\end{defn} 

The following theorem is the main result of this section. We recall some notations. We extend the adjoint pair $(F=B\otimes_A-, G={\rm Hom}_B(B,-))$ to an adjoint pair $(\mathbf{K}(F),\mathbf{K}(G))$ between $\mathbf{K}^*(A\mbox{-mod})$ and $\mathbf{K}^*(B\mbox{-mod})$ in a natural manner, where $*$ can be $b$, $-$ or $+$. Since $G$ is exact, $$\mathbf{D}(G)=\mathbf{K}(G)\colon\mathbf{D}^-(B\mbox{-{\rm mod}})\rightarrow\mathbf{D}^-(A\mbox{-{\rm mod}})$$ is its own right derived functor. Since $F$ is right exact and preserves projectives, the left derived functor of $F$ is $$\mathbb{L}F=q\mathbf{K}(F)p\colon\mathbf{D}^-(A\mbox{-{\rm mod}})\rightarrow\mathbf{D}^-(B\mbox{-{\rm mod}}),$$ where $q$ is the localization functor and sometimes we omit it on objects.
\begin{thm}\label{sep}
	Let $\phi\colon A\rightarrow B$ be a $k$-algebra extension between two finite dimensional $k$-algebras with $(F,G)$ the corresponding adjoint pair. \\
	{\rm (1)} The extension $\phi$ is a split extension if and only if $\mathbf{K}(F)$ is separable if and only if $\mathbb{L}F$ is separable.\\
	{\rm (2)} The extension $\phi$ is a separable extension if and only if $\mathbf{K}(G)$ is separable, which are implied by that $\mathbf{D}(G)$ is separable. If further $F$ is exact, then $\phi$ is a separable extension if and only if $\mathbf{D}(G)$ is separable.
\end{thm}

To prove the theorem, we need some preparation. We first give some examples. The proof is in the end of this section.

\begin{exm}\label{exm}
	Let $A$ be a $k$-algebra.\\
	{\rm (1)} For each two-sided ideal $I$ of $A$, the canonical quotient $A\rightarrow A/I$ is separable.\\
	{\rm (2)} Let $G$ be a finite group acting on $A$ with its order $|G|$ invertible in $k$. Then the extension from $A$ to its skew group algebra $AG$ is separable and split; see \cite[Section 1]{RR}.\\
	{\rm (3)} Let $K/k$ be a finite field extension. We consider the extension $$\phi\colon A\rightarrow A\otimes_kK, \psi(a)=a\otimes 1.$$ It has a retraction $a\otimes\lambda\mapsto a\pi(\lambda),\forall a\in A, \lambda\in K,$ as an $A$-bimodules homomorphism, where $\pi:K\rightarrow k$ is a $k$-linear retraction of $k\hookrightarrow K$. Hence $\phi$ is a split extension by \cite[Proposition~1.3 (2)]{NBO}.
	
	If further $K/k$ is separable, then the multiplication map $K\otimes_kK\rightarrow K$ has a section $\psi$ as a $K$-bimodule homomorphism. It induces an $A\otimes_kK$-bimodule homomorphism
	$$A\otimes_kK\overset{\mathrm{Id}_A\otimes\psi}{\rightarrow} A\otimes_kK\otimes_kK\overset{\theta\otimes \mathrm{Id}_{K\otimes_kK}}{\rightarrow} A\otimes_AA\otimes_kK\otimes_kK\overset{\rho}{\rightarrow} (A\otimes_kK)\otimes_A(A\otimes_kK),$$ where $\theta(a)=a\otimes1$ and $\rho(a_1\otimes a_2\otimes \lambda_1\otimes\lambda_2)=a_1\otimes \lambda_1\otimes a_2\otimes\lambda_2$, $\forall a, a_1, a_2\in A, \lambda_1,\lambda_2\in K$. This is a section of the multiplication map $(A\otimes_kK)\otimes_A(A\otimes_kK)\rightarrow A\otimes_kK$. By \cite[Proposition~1.3 (1)]{NBO}, $\phi$ is also a separable extension as $k$-algebras.
\end{exm}

\begin{lem}
	Let $F$ be a functor between two module categories.\\
	{\rm (1)} The functor $F$ is separable if and only if so is $\mathbf{K}(F)$. \\
	{\rm (2)} If $F$ is exact and $\mathbf{D}(F)$ is separable, then $F$ is separable.
\end{lem}
\begin{proof}
	(1). If $F$ is separable, it has a natural retraction $H_{M,N}$ on morphisms for any modules $M$ and $N$. For any complexes $X=(X^i)_{i\in\mathbb{Z}}$ and $Y=(Y^i)_{i\in\mathbb{Z}}$, we extend $H_{M,N}$ term-wise to a natural retraction $$H_{X,Y}\colon (f^i)_{i\in\mathbb{Z}}\mapsto(H_{X^i,Y^i}(f^i))_{i\in\mathbb{Z}}$$ on chain maps between complex categories. Moreover, if $(f^i)_{i\in\mathbb{Z}}$ is null-homotopic with $f^i=F(d^{i-1})\circ s^i+s^{i+1}\circ F(d^i)$, then $(H_{X^i,Y^i}(f^i))_{i\in\mathbb{Z}}$ is null-homotopic with $$H_{X^i,Y^i}(f^i)=d^{i-1}\circ H_{X^i,Y^{i-1}}(s^i)+H_{X^{i+1},Y^i}(s^{i+1})\circ d^i$$ due to the naturality of $H_{X^i,Y^j}$. Thus we get a natural retraction on morphisms for $\mathbf{K}(F)$ in homotopy categories.
	
	Assume that $\mathbf{K}(F)$ is separable. For any two $A$-modules $M$ and $N$, viewing as stalk complexes at degree zero, we identify ${\rm Hom}_{A}(M,N)$ with ${\rm Hom}_{\mathbf{K}^*(A\mbox{-{\rm mod}})}(M,N)$ and ${\rm Hom}_{B}(F(M),F(N))$ with ${\rm Hom}_{\mathbf{K}^*(B\mbox{-{\rm mod}})}(\mathbf{K}(F)(M),\mathbf{K}(F)(N))$. Hence a natural retraction 
	$${\rm Hom}_{\mathbf{K}^*(B\mbox{-{\rm mod}})}(\mathbf{K}(F)(M),\mathbf{K}(F)(N))\rightarrow{\rm Hom}_{\mathbf{K}^*(A\mbox{-{\rm mod}})}(M,N),$$ for $\mathbf{K}(F)$ gives a natural retraction $${\rm Hom}_{B}(F(M),F(N))\rightarrow{\rm Hom}_{A}(M,N)$$ for $F$.
	
	(2). For any two modules $M$ and $N$ in $A\mbox{-mod}$, we have natural isomorphisms $${\rm Hom}_{A}(M,N)\simeq{\rm Hom}_{\mathbf{D}^*(A\mbox{-{\rm mod}})}(M,N)$$ and $${\rm Hom}_{B}(F(M),F(N))\simeq{\rm Hom}_{\mathbf{D}^*(B\mbox{-{\rm mod}})}(F(M),F(N)),$$ where $M$, $N$, $F(M)=\mathbf{D}(F)(M)$, $F(N)=\mathbf{D}(F)(N)$ in derived categories are viewed as stalk complexes at degree zero. Hence the statement holds.
\end{proof}

When consider separable functors in adjoint pairs, the following lemma is used frequently; see \cite[1.2]{R}
\begin{lem}\label{sepeq} 
	Let $(F,G,\eta\colon {\rm Id}_\mathcal{C}\rightarrow GF,\epsilon\colon FG\rightarrow{\rm Id}_\mathcal{D})$ be an adjoint pair between categories $\mathcal{C}$ and $\mathcal{D}$. Then the following statements hold.\\
	{\rm (1)} The functor $F$ is separable if and only if there is a natural transformation $\delta\colon GF\rightarrow{\rm Id}_\mathcal{C}$ such that $\delta\circ\eta={\rm Id}$.\\
	{\rm (2)} The functor $G$ is separable if and only if there is a natural transformation $\zeta\colon {\rm Id}_\mathcal{D}\rightarrow FG$ such that $\epsilon\circ\zeta={\rm Id}$.	
\end{lem}

Let $(F,G)$ be an adjoint pair between $A\mbox{-mod}$ and $B\mbox{-mod}$ with unit $\eta$ and counit $\epsilon$. Assume that $G$ is an exact functor. The unit and counit of $(\mathbf{K}(F),\mathbf{K}(G))$ are $\mathbf{K}(\eta)$ and $\mathbf{K}(\epsilon)$, where $$\mathbf{K}(\eta)_X=\overline{(\eta_{X^i})_{i\in\mathbb{Z}}},\forall X\in\mathbf{K}^*(A\mbox{-mod})\mbox{ and }  \mathbf{K}(\epsilon)_Y=\overline{(\epsilon_{Y^i})_{i\in\mathbb{Z}}},\forall Y\in\mathbf{K}^*(B\mbox{-mod}).$$

It is well-known that $(\mathbb{L}F,\mathbf{D}(G))$ is an adjoint pair; see \cite[Section 10.7.1]{W}. A natural isomorphism $\Psi$ of this adjoint pair can be given by the following commutative diagram for any complexes $X$ and $Y$.
$$\xymatrix@1{{\rm Hom}_{\mathbf{D}^-(B\mbox{-{\rm mod}})}(\mathbb{L}F(X),Y)\ar[r]^-{\Psi}\ar@{=}[d] & {\rm Hom}_{\mathbf{D}^-(A\mbox{-{\rm mod}})}(X,\mathbf{D}(G)(Y))\\
{\rm Hom}_{\mathbf{D}^-(B\mbox{-{\rm mod}})}(\mathbf{K}(F)(pX),Y)\ar[d]^-{q^{-1}}& {\rm Hom}_{\mathbf{D}^-(A\mbox{-{\rm mod}})}(pX,\mathbf{D}(G)(Y))\ar[u]^-{f}\\
{\rm Hom}_{\mathbf{K}^-(B\mbox{-{\rm mod}})}(\mathbf{K}(F)(pX),Y)\ar[r]^-{\psi}&{\rm Hom}_{\mathbf{K}^-(A\mbox{-{\rm mod}})}(pX,\mathbf{K}(G)(Y))
\ar[u]^-{q}}$$ 
Here, since $pX$ and $\mathbf{K}(F)(pX)$ are complexes of projectives, the localization functors $q$ are fully faithful. The isomorphism $f$ is induced by $a_X\colon pX\rightarrow X$, the projective resolution in $\mathbf{K}^-(A\mbox{-{\rm mod}})$, which is a quasi-isomorphism and natural on $X$. Finally, $\psi$ is the natural isomorphism of the adjoint pair $(\mathbf{K}(F),\mathbf{K}(G);\mathbf{K}(\eta),\mathbf{K}(\epsilon))$.

Using $\Psi$, we obtain the unit $\mathbf{D}(\eta)$ and counit $\mathbf{D}(\epsilon)$ of $(\mathbb{L}F,\mathbf{D}(G))$, where $$\mathbf{D}(\eta)_{X}=(\mathbf{K}(\eta)_{pX}) a_X^{-1}\colon X\longrightarrow \mathbf{D}(G)(\mathbb{L}F(X)), \forall X\in\mathbf{D}^-(A\mbox{-{\rm mod}})$$ are given by fractions (see \cite[Section 10.3]{W}) and $$\mathbf{D}(\epsilon)_{Y}=q(\mathbf{K}(\epsilon)_Y\circ\mathbf{K}(F)(a_{\mathbf{D}(G)(Y)}))\colon\mathbb{L}F(\mathbf{D}(G)(Y))\longrightarrow Y,\forall Y\in\mathbf{D}^-(B\mbox{-{\rm mod}}).$$ They are natural since they are constructed by functors and natural transformations.

The following lemma is well-known. We compare \cite[3.1]{V}.

\begin{lem}
	Let $(F,G)$ be an adjoint pair between module categories with $G$ an exact functor. Then $F$ is separable if and only if $\mathbb{L}F$ is separable.
\end{lem}
\begin{proof}
	By Lemma 3.4 and Lemma 3.5, we assume that $\mathbf{K}(F)$ is separable with $K(\delta)$ a natural retraction of $\mathbf{K}(\eta)$. For each $X\in\mathbf{D}^-(A\mbox{-{\rm mod}})$, $$q(a_X\circ K(\delta)_{pX})\circ(\mathbf{K}(\eta)_{pX})a_X^{-1}=(a_X\circ K(\delta)_{pX}\circ\mathbf{K}(\eta)_{pX})a_X^{-1}=\mathrm{Id}_X.$$ Hence the unit $\mathbf{D}(\eta)$ of $(\mathbb{L}F,\mathbf{D}(G))$ has a retraction $q(a_X\circ K(\delta)_{pX})$, which is natural on $X$. Therefore, $\mathbb{L}F$ is separable. 
	
	Conversely, if $\mathbb{L}F$ is separable, $\mathbf{D}(\eta)$ has a natural retraction $D(\delta)$.  For each $X\in A\mbox{-{\rm mod}}$, viewing as a stalk complex at degree zero, we have $$\mathbf{K}(G)(\mathbf{K}(F)(a_X))\circ\mathbf{K}(\eta)_{pX}=\mathbf{K}(\eta)_X\circ a_X\mbox{ and }H^0(q(\mathbf{K}(\eta)_X))=\eta_X,$$ where $H^0\colon\mathbf{D}^-(A\mbox{-{\rm mod}})\rightarrow A\mbox{-{\rm mod}}$ is the cohomology functor at degree zero. Since $G$ and $F$ are right exact and $a_X$ is a quasi-isomorphism, $H^0(q\mathbf{K}(G)(\mathbf{K}(F)(a_X)))$ and $H^0(q(a_X))$ are invertible. Hence $$H^0(q\mathbf{K}(G)(\mathbf{K}(F)(a_X)))^{-1}\circ H^0(q(\mathbf{K}(\eta)_X))= H^0(q(\mathbf{K}(\eta)_{pX}))\circ H^0(q(a_X))^{-1}.$$ For each $X\in A\mbox{-{\rm mod}}$,
	\begin{align*}
	&H^0(D(\delta)_X)\circ H^0(q\mathbf{K}(G)(\mathbf{K}(F)(a_X)))^{-1}\circ\eta_X\\
	=&H^0(D(\delta)_X)\circ H^0(q\mathbf{K}(G)(\mathbf{K}(F)(a_X)))^{-1}\circ H^0(q(\mathbf{K}(\eta)_X))\\
	=&H^0(D(\delta)_X)\circ H^0(q(\mathbf{K}(\eta)_{pX}))\circ H^0(q(a_X))^{-1}\\
	=&H^0(D(\delta)_X\circ\mathbf{K}(\eta)_{pX}a_X^{-1})=H^0(D(\delta)_X\circ\mathbf{D}(\eta)_{X})\\
	=&H^0(\mathrm{Id}_{X})=\mathrm{Id}_X.
	\end{align*} Therefore, we obtain a retraction $H^0(D(\delta)_X)\circ H^0(q\mathbf{K}(G)(\mathbf{K}(F)(a_X)))^{-1}$ of $\eta_{X}$ which is natural on $X$. So $F$ is separable by Lemma $\ref{sepeq}$.
\end{proof}

Dually, if $F$ is exact, $(\mathbf{D}(F),\mathbb{R}G)$ is an adjoint pair between $\mathbf{D}^+(A\mbox{-{\rm mod}})$ and $\mathbf{D}^+(B\mbox{-{\rm mod}})$. We have the following result.
\begin{lem}
	Let $(F,G)$ be an adjoint pair between module categories with $F$ an exact functor. Then $G$ is separable if and only if $\mathbb{R}G$ is separable. 
\end{lem}
\begin{rem}
	In the above lemma, the condition that $F$ is an exact functor is necessary; see Example \ref{cntemp}.
\end{rem}

\begin{proof}[Proof of Theorem 3.2]
	The statement (1) holds by Lemma 3.4 (1) and Lemma~3.6.
    The statement (2) is a consequence of Lemma 3.4 (1), (2) and Lemma 3.7.
\end{proof}

\section{derived-discreteness and split/separable extensions}
We keep the notation as in Section 3. The following main result shows that derived-discreteness is compatible with split/separable extensions.
\begin{thm}\label{main}
	Let $\phi\colon A\rightarrow B$ be a $k$-algebra extension between two finite dimensional $k$-algebras. Then following statements hold.\\
	{\rm (1)} If $\phi$ is a split extension and $B$ is derived-discrete over $k$, then $A$ is derived-discrete over $k$.\\
	{\rm (2)} If $\phi$ is a separable extension with $B_A$ a projective right $A$-module and $A$ is derived-discrete over $k$, then $B$ is derived-discrete over $k$.
\end{thm}
\begin{proof}
	For (1), if $A$ is not derived-discrete, by Lemma \ref{spequ} there is an $\underline{n}\in\mathbb{N}^{(\mathbb{Z})}$ such that $$\{[P]\in\mathbf{K}^b(A\mbox{-proj})\,|\,\mbox{c-}\underline{\dim}_kP=\underline{n}\}$$ is an infinite set. Since $\dim_k\mathbf{K}(F)(P)^i=\dim_kB\otimes_AP^i$ for each $i$, the set $$\{\mbox{c-}\underline{\dim}_k\mathbf{K}(F)(P)\,|\,\underline{\dim}_kP=\underline{n}\}$$ is bounded, say by $\underline{m}\in\mathbb{N}^{(\mathbb{Z})}$. By the derived-discreteness of $B$ and Lemma \ref{spequ}, $$\{[\mathbf{K}(F)(P)]\in\mathbf{K}^b(B\mbox{-proj})\,|\,\mbox{c-}\underline{\dim}_k\mathbf{K}(F)(P)\leq\underline{m}\}$$ is a finite set. Therefore, $$\{[\mathbf{K}(F)(P)]\in\mathbf{K}^b(B\mbox{-proj})\,|\,\mbox{c-}\underline{\dim}_kP=\underline{n}\}$$ is a finite set. Then $$\{[\mathbf{K}(G) (\mathbf{K}(F)(P))]\in\mathbf{K}^b(A\mbox{-mod})\,|\,\mbox{c-}\underline{\dim}_kP=\underline{n}\}$$ is a finite set. 
	
	Since $\phi$ is a split extension, $\mathbf{K}(F)$ is a separable functor by Theorem \ref{sep}. By Lemma~\ref{sepeq} each $P\in\mathbf{K}^b(A\mbox{-proj})$ is a direct summand of $\mathbf{K}(G)(\mathbf{K}(F)(P))$ in $\mathbf{K}^b(A\mbox{-mod})$. Then the first and last sets above imply that there is an object $\mathbf{K}(G)(\mathbf{K}(F)(P))$ in $\mathbf{K}^b(A\mbox{-mod})$ with infinitely many pairwise non-isomorphic direct summands, which is impossible as $\mathbf{K}^b(A\mbox{-mod})$ is Krull-Schmidt.	
	
	For (2), it is a consequence of Theorem \ref{sep} (2) and the following lemma.
\end{proof}
\begin{lem}
	If $\phi\colon A\rightarrow B$ is a $k$-algebra extension with $\mathbf{D}(G)$ separable and $A$ is derived-discrete over $k$, then $B$ is derived-discrete over $k$.
\end{lem}
\begin{proof}
	If $B$ is not derived-discrete, there is an $\underline{n}=(n_i)_{i\in\mathbb{Z}}\in\mathbb{N}^{(\mathbb{Z})}$ such that $$\{[X]\in\mathbf{D}^b(B\mbox{-mod})\,|\,\underline{\dim}_kX=\underline{n}\}$$ is an infinite set. Since $\underline{\dim}_k\mathbf{D}(G)(X)=\underline{\dim}_kX$ and $A$ is derived-discrete, $$\{[\mathbf{D}(G)(X)]\in\mathbf{D}^b(A\mbox{-mod})\,|\, \underline{\dim}_kX=\underline{n}\}$$ is a finite set. Then $$\{[\mathbb{L}F(\mathbf{D}(G)(X))]\in\mathbf{D}^-(B\mbox{-mod})\,|\,\underline{\dim}_kX=\underline{n}\}$$ is a finite set. 
	
	Since $\mathbf{D}(G)$ is separable, Lemma \ref{sepeq} implies that each $X\in\mathbf{D}^b(B\mbox{-mod})$ is a direct summand of $\mathbb{L}F(\mathbf{D}(G)(X))$ in $\mathbf{D}^-(B\mbox{-mod})$. Then the first and last sets above imply that there is an object $Y$ in $\mathbf{D}^-(B\mbox{-mod})$ with infinitely many pairwise non-isomorphic direct summands in the first set above. Let $t$ be the least number such that $n_t\neq0$. Denote by $\tau_{\geq t}Y\in\mathbf{D}^b(B\mbox{-mod})$ the good truncation of $Y$ at degree $t$. Each direct summand of $Y$ in the first set above is still a direct summand of $\tau_{\geq t}Y$. So $\tau_{\geq t}Y$ has infinitely many pairwise non-isomorphic direct summands. It is impossible since $\mathbf{D}^b(B\mbox{-mod})$ is Krull-Schmidt.
\end{proof}
\begin{rem}
	One can consider the derived Brauer-Thrall conjecture on algebra extensions; see \cite{LZ} for the field extension case. The above theorem may also hold for derived-tameness. If gl.dim$A<\infty$, it has been proved that the derived-tameness of $B$ implies that of $A$; see \cite[Theorem~3.1]{Z}.
\end{rem}
In (2) of the above theorem, the condition that $B$ is a projective right $A$-module is necessary. 

\begin{exm}\label{cntemp}
	Let $k$ be algebraically closed, and $Q$ be a quiver as
	$$
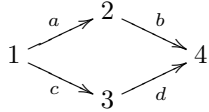
\xymatrix@R=1ex{ & 2 \ar[dr]^b & \\1\ar[ur]^a\ar[dr]_c& &4.\\  &3\ar[ur]_d& }$$\\
	Consider the quotient $kQ/\langle ba\rangle\twoheadrightarrow kQ/\langle ba,dc\rangle$. It is a separable extension. We have that $kQ/\langle ba\rangle$ is derived-discrete. But $kQ/\langle ba,dc\rangle$ is iterated tilted of $\tilde{A}$ type, which is not derived-discrete; see \cite[2.1 and 2.2]{V}. In this case, $\mathbf{D}(G)$ is not separable, otherwise $kQ/\langle ba,dc\rangle$ is derived-discrete by the lemma above.
\end{exm}

\section{piecewise hereditary algbras and split/separable extensions}
We give an analogous statement for piecewise hereditary algebras in this section. For the field extension case, we have a refined result; see \cite{L}.

Recall that an algebra is called \emph{piecewise hereditary} of type $H$ if it is derived equivalent to $\mathbf{D}^b(H)$ for a hereditary abelian category $H$. When $k$ is algebraically closed, recall that one class of derived-discrete algebras is the piecewise hereditary algebras of Dynkin type. In view of Theorem \ref{main}, it is natural to  expect that piecewise hereditary algebras is compatible with split/separable extension.
\begin{prop}\label{ph}
	Let $\phi\colon A\rightarrow B$ be a $k$-algebra extension between two finite dimensional $k$-algebras. The following statements hold.\\
	{\rm (1)} If $\phi$ is a split extension and $B$ is piecewise hereditary, then $A$ is piecewise hereditary.\\
	{\rm (2)} If $\phi$ is a separable extension with $_AB$ a projective left $A$-module and $A$ is piecewise hereditary, then $B$ is piecewise hereditary.
\end{prop}
Recall that the \textit{strong global dimension} of a $k$-algebra $A$, denoted by $\mbox{s.gl.}\dim A$, is given by $$\sup\{l(P)\,|\,0\neq P\in\mathbf{K}^b(A\mbox{-proj})\text{ indecomposable and homotopically-minamal}\},$$ where $l(P)=\min\{b-a\,|\,a,b\in\mathbb{Z},b\geq a,\mbox{ and } P^i=0\mbox{ for }i<a\mbox{ and }i>b\}$ is the length of $P\neq0$.

We have a homological characterization of piecewise hereditary algebras saying that $A$ is piecewise hereditary if and only if $\mbox{s.gl.}\dim A$ is finite; see \cite[Theorem~3.2]{H}.
\begin{proof}
	For (1), we claim that $\mbox{s.gl.}\dim A\leq\mbox{s.gl.}\dim B$. Indeed, for each indecomposable $P$ in $\mathbf{K}^b(A\mbox{-proj})$, by Lemma \ref{sepeq}, $P$ is a direct summand of $\mathbf{K}(G)(\mathbf{K}(F)(P))$ in $\mathbf{K}^b(A\mbox{-mod})$. The length of each direct summand of $\mathbf{K}(G)(P)$ in $\mathbf{K}^b(B\mbox{-proj})$ is not larger than $\mbox{s.gl.}\dim B$. As $l(\mathbf{K}(G) (\mathbf{K}(F)(P)))=l(\mathbf{K}(F)(P))$ and each indecomposable direct summand will not have larger length, we have that $l(P)\leq\mbox{s.gl.}\dim B$.
	
	For (2), the condition that $B$ is a projective left $A$-module makes $G$ sending projectives to projectives. So we can prove that $\mbox{s.gl.}\dim B\leq\mbox{s.gl.}\dim A$ similarly as above.
\end{proof}

In the proof of Proposition \ref{ph} (2), the condition that $_AB$ is projective is necessary.

\begin{exm}
	Let $k$ be algebraically closed, and $Q$ be a quiver as
	$$
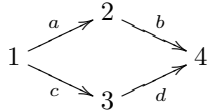
\xymatrix@R=1ex{ & 2 \ar[dr]^b & \\1\ar[ur]^a\ar[dr]_c& &4\\  &3\ar[ur]_d& }$$\\
    Consider the quotient $kQ\twoheadrightarrow kQ/\langle ba\rangle$, which is a separable algebra extension. We have that $kQ$ is (piecewise) hereditary, but $kQ/\langle ba\rangle$ is derived-discrete but not piecewise hereditary according to the classification of derived-discrete algebras in \cite{V} (we recall it in the next section).
\end{exm}

\section{Applications}
We give two applications of the results in Section 4 and Section 5 for field extensions and skew group algebra extensions.

We recall from \cite[2.1 Theorem]{V} the classification of derived-discrete algebras over an algebraically closed field $k$: a connected $k$-algebra $A$ is derived-discrete over $k$ if and only if $A$ is either piecewise hereditary of Dynkin type or $A$ is Morita equivalent to $kQ/I$ such that $kQ/I$ is gentle one-cycle with the clock condition, that is, $kQ/I$ is a gentle algebra containing exactly one cycle, and in the cycle the number of clockwise oriented relations does not equal the number of counterclockwise oriented relations. 

\begin{prop}
	Let $K/k$ be a finite separable field extension and $A$ be a finite dimensional $k$-algebra. Then the following statements hold.\\
	{\rm (1)} The algebra $A$ is derived-discrete over $k$ if and only if $A\otimes_kK$ is derived-discrete over $K$. \\
	{\rm (2)} The algebra $A$ is piecewise hereditary if and only if so is $A\otimes_kK$.
\end{prop}
\begin{proof}
	By Example \ref{exm} 3), the extension $A\rightarrow A\otimes_kK$ is both split and separable and $A\otimes_kK$ is a left and right projective $A$-module. 
	
	(1) By Theorem \ref{main}, $A$ is derived-discrete over $k$ if and only if $A\otimes_kK$ is derived-discrete over $k$. By Lemma 2.3, $A\otimes_kK$ is derived-discrete over $k$ if and only if $A\otimes_kK$ is derived-discrete over $K$. 
	
	(2) By Proposition \ref{ph}.
\end{proof}

Let $A$ be a finite dimensional algebra over an algebraically closed field $k$. Assume that $G$ is a finite group acting on $A$ with its order $|G|$ invertible in $k$. The algebra extension from $A$ to its skew group algebra $AG$ is both split and separable extension with $AG$ a both left and right projective $A$-module; see Example~\ref{exm} 2).

\begin{prop}\label{last}
	Let $A$ be a connected algebra and $A\rightarrow AG$ be a skew group extension as above. Then the following statements hold.\\
	{\rm (1)} The algebra $A$ is derived-discrete if and only if so is $AG$.\\
	{\rm (2)} The algebra $A$ is piecewise hereditary of Dynkin type if and only if so is each connected component of $AG$.\\
	{\rm (3)} The algebra $A$ is Morita equivalent to a gentle one-cycle algebra with the clock condition if and only if so is each connected component of $AG$.
\end{prop}
\begin{proof}
	(1) is a consequence of Theorem \ref{main}. 
	
	An algebra is derived-discrete (or piecewise hereditary) if and only if so are its connected components. Notice that a gentle one-cycle algebra with the clock condition is not piecewise hereditary; see \cite{V}.
	
	For the ``only if'' part of (2). By the classification of derived-discrete algebras, if $A$ is piecewise hereditary of Dynkin type, then $A$ is derived-discrete. Hence each connected component of $AG$ is derived-discrete and piecewise hereditary by (1) and Proposition~\ref{ph}. Therefore, it must be piecewise hereditary of Dynkin type. 
	
	The ``if'' part of (2) and statement (3) can be proved in a similar argument.
\end{proof}

\vskip 10pt

\noindent {\bf Acknowledgements.}\quad The author is grateful to Prof. Xiao-Wu Chen and Chao Zhang for their suggestions. He also thanks the referees. This work is supported by the National Natural Science Foundation of China (No.12171297 and No.12201166).

\bibliography{}

\begin{thebibliography}{999}

\bibitem{B}{\sc R. Bautista}, {\em On derived tame algebras}, Bol. Soc. Mat. Mexicana {\bf 13} (3) (2007), 25--54.

\bibitem{CYZ} {\sc X.W. Chen, Y. Ye, and P. Zhang}, {\em Algebras of derived dimension zero}, Comm. Algebra {\bf 36}(1) (2008), 1--10.

\bibitem{BGS} {\sc G. Bobinski, C. Geiss, and A. Skowronski}, {\em Classification of derived discrete algebras}, Cent. Euro. J. Math. {\bf 2}(1) (2004), 19--49.

\bibitem{H} {\sc D. Happel, and D. Zacharia}, {\em A homological characterisation of piecewise hereditary algebras},
Math. Z. {\bf 260} (1) (2008), 177--185.

\bibitem{L}{\sc J. Li}, {\em Piecewise hereditary algebras under field extensions}, Czechoslovak Math. J. {\bf 71} (4) (2021), 1025--1034. 

\bibitem{LZ}{\sc J. Li, and C. Zhang}, {\em Derived representation type and field extensions}, Colloq. Math. {\bf 168} (1) (2022), 105--117. 

\bibitem{NBO} {\sc C. Năstăsescu, M. Van den Bergh, and F. Van Oystaeyen}, {\em Separable functors applied to graded rings},
J. Algebra {\bf 123} (1989), 397--413.

\bibitem{NPP} {\sc N. Broomhead, D. Pauksztello, and D. Ploog}, {\em
Discrete derived categories I: homomorphisms, autoequivalences and t-structures},
Math. Z. {\bf 285}(1-2) (2017), 39--89.

\bibitem{R} {\sc M. D. Rafael}, {\em Separable functors revisited}, Comm. Algebra {\bf 18} (1990), 1445--1459.

\bibitem{RR}  {\sc I. Reiten, and C. Riedtmann}, {\em Skew group algebras in the representation theory of Artin algebras}, J. Algebra {\bf 92} (1985), 224--282. 

\bibitem{V} {\sc D. Vossieck}, {\em The algebras with discrete derived category}, J. Algebra {\bf 243} (2001), 168--176.

\bibitem{W} {\sc C. A. Weibel}, {An introduction to homological algebra},
 Cambridge university press {\bf 38} (1995) 

\bibitem{Z} {\sc C. Zhang}, {\em Derived representation type and cleaving functors}, Comm. Algebra {\bf 46}(6) (2018), 2696--2701.

\end{thebibliography}

\vskip 10pt

{\footnotesize \noindent Jie Li\\
	School of Mathematics, Hefei University of Technology, Hefei 230000, Anhui, PR China}
\end{document}